\documentclass[11pt,reqno]{amsproc} 
\usepackage{amstext,amsmath,amssymb,amsthm}

\usepackage[T1]{fontenc}
 
\usepackage{tikz}
\usepackage{tikz-3dplot}

\usepackage{latexsym}
\usepackage{exscale}
\usepackage{color}
 
\newcommand{\Primes}{{ \mathbb  P  }}

\newcommand{\C}{  \mathbb  C }

\newcommand{\N}{  \mathbb N }

\newtheorem{Thm}{Theorem}[section]
\newtheorem{Lemma}[Thm]{Lemma}
\newtheorem{Cor}[Thm]{Corollary}
\newtheorem{Prop}[Thm]{Proposition}
\newcommand{\dsize}{\displaystyle}
\newcommand{\cal}{\mathcal}
\numberwithin{equation}{section}
\newcommand{\E} {{\mathcal E}}

\renewcommand{\H} {\mathcal H} 
\renewcommand{\P} {\mathcal P}

\newcommand{\Mod}[1]{\, (\mathrm{mod}\, #1)}

\theoremstyle{remark}

\newtheorem{Rem}{Remark}[section]

\newcommand{\BEL}{\begin{equation}\label}
	\newcommand{\EE}{\end{equation}}

\begin{document}
	
	\title[Cyclotomy of symmetric polynomials] {\large  Cyclotomy of symmetric polynomials}

	\author{L. De Carli}
	\address{Laura De Carli, Department of Mathematics, Florida International University,   Miami, FL 33199, USA.}
	\email{decarlil@fiu.edu}

	\author{M.   Laporta}
	\address{Maurizio Laporta, Dipartimento di Matematica e Applicazioni, Università degli Studi di Napoli  "Federico II", 80126 Napoli, Italy}
	\email{mlaporta@unina.it }
	
 \keywords{cyclotomic polynomials,   homogeneous symmetric polynomials,  roots of unity}
	
	\subjclass 
	{	05E05, %: Symmetric functions and polynomials.
		11C08, %: Polynomials, location of zeros (Number theory context)
		12D10. %: Polynomials: location of zeros (algebraic theorems)}
}

\maketitle

\begin{abstract}
	We prove some new identities for homogeneous  symmetric  polynomials  when   evaluated at  roots of  unity. 
Some of these formulas are applied to obtain new insights onto  cyclotomic polynomials. 
\end{abstract}

\section{Introduction}\label{introduction}

 The word cyclotomy means literally “circle-cutting” and historically refers to the problem of dividing the unit circle into a given number of arcs of equal lengths. 
 This problem was studied more than two thousand years ago by the Greek geometers and, following the work of Gauss, evolved into the modern theory of cyclotomic fields.  
   \cite[Ch.\ 2]{D}.
 
 For $n\in\mathbb N:=\{1,2,\ldots\}$, we let
 $$
 \mathcal Z_n:=\{1,\zeta_n,\ldots,\zeta_n^{n-1}\},
 \qquad \zeta_n:=e^{2\pi i/n},
$$
 be the set of $n$th roots of unity.
 By cyclotomy of a function we mean the study of its values
 at roots of unity.  A central role is played by the cyclotomic polynomial, defined for every $n\in\N$ as
$$
 \Phi_n(z):=\prod_{\zeta_n^j\in\mathcal Z_n^*}(z-\zeta_n^j),
$$
 where $\mathcal Z_n^*:=\{\zeta_n^j:0\leq j<n,\ \gcd(j,n)=1\}
 $ is the set
of the primitive $n$th roots of unity.
 The factorization
 $ \dsize
 z^n-1=\prod_{d\mid n}\Phi_d(z)
$ 
 connects vanishing at roots of unity with divisibility by
 cyclotomic polynomials.

  \medskip
 
   In this paper we continue our investigation of the cyclotomy
   of complete homogeneous symmetric polynomials initiated in
   \cite{DEL,DL}. Further, we establish new identities and give applications to
   cyclotomic polynomials.

The elementary homogeneous symmetric
polynomial of degree $m\in\N$ in the variables $z_1,\ldots, z_r$ is 
$$
\E_m(z_1,\ldots,z_r):=
\sum_{1\leq j_1<\cdots<j_m\leq r}z_{j_1}\cdots z_{j_m},
$$
while 
$$
\H_m(z_1,\ldots,z_r)
:=\sum_{1\leq j_1\leq\cdots\leq j_m\leq r}z_{j_1}\cdots z_{j_m}
$$
is called complete homogeneous symmetric
polynomial of degree $m$ in the variables $z_1,\ldots, z_r$.
Since the sum that defines    $\E_m$ is empty when $r< m$, we let $\E_m(z_1,\ldots,z_r)=0$ in this case; we also define $\E_0(z_1,\ldots,z_r)=\mathcal H_0(z_1,\ldots,z_r)=1$. 

Both polynomials $\E_m$ and $\H_m$ are invariant under permutations of the
variables; we therefore write $\mathcal E_m(J)$ and
$\mathcal H_m(J)$ when $J$ is a finite multiset of complex numbers, i.e., a set where elements may appear with repetitions.
Let also adopt the convention
\begin{equation}\label{convention1}
	\mathcal E_m(\emptyset)=\mathcal H_m(\emptyset)
	:=\begin{cases}
		1,&m=0,\cr
		0,&m>0.
	\end{cases}
\end{equation}
We refer to \cite{Mcdonald,St} for the  basic properties of the  symmetric polynomials  $\E_m$ and $\H_m$. In particular,   \cite{St} is an excellent reference for an introduction to multisets.

The classical Vieta formulas express the coefficients of a complex polynomial  
\begin{equation}\label{polynomial}
	\P(z)=a_0\prod_{j=1}^m(z-z_j)
	=\sum_{k=0}^m a_kz^{m-k},
	\qquad a_0\neq0,
\end{equation}
in terms of its roots.  Specifically,
\begin{equation}\label{e-classVieta}
	\frac{a_k}{a_0}=(-1)^k\E_k(z_1,\ldots,z_m),
	\qquad 0\leq k\leq m.
\end{equation}
See \cite[Ex.\ 4.6.6]{B}.
Applying them to the polynomial $z^n-1$ gives
\begin{equation}\label{e-classVietaz^n-1}
	\E_k(\mathcal Z_n)=
	\begin{cases}
		1,&k=0,\\
		(-1)^{n+1},&k=n,\\
		0,&\text{otherwise}.
	\end{cases}
\end{equation}
For the complete homogeneous symmetric polynomials,
\cite[Cor.\ 1.5]{DEL} gives
\begin{equation}\label{1e-spec-case}
	\H_k(\mathcal Z_n)=
	\begin{cases}
		1,&n\mid k,\\
		0,&\text{otherwise},
	\end{cases}
	\qquad k\in\mathbb N_0.
\end{equation}
See Lemma \ref{T-cyclic-1} below for a more general property.

Our  first result extends the latter evaluation to every
set of $n$th roots of unity.

 \begin{Thm}\label{T1-main} Let $m \in\N_0$. For every nonempty $J\subseteq {\cal Z}_n$ one has
 	\begin{equation}\label{1main}
 		\H_m(J)= (-1)^{\{m\}_n}\E_{\{m\}_n}({\cal Z}_n\setminus J),
 	\end{equation}
 	where  $\{m\}_n$  denotes the remainder of the division of $m$ by $n\in\N$.
 \end{Thm}
  The proof is given in \S \ref{Ths}.  In view of  \eqref{convention1}, the identity \eqref{1e-spec-case} is recovered by   taking $J={\cal Z}_n$
 in \eqref{1main}. Further,   observe that \eqref{1main} is false 
  when $J=\emptyset$ and $n$ divides $m\not=0$ because of  \eqref{convention1} and \eqref{e-classVietaz^n-1}, while the identity 
 \begin{equation}\label{1mainempty}
 	\H_{\{m\}_n}(J)= (-1)^{\{m\}_n}\E_{\{m\}_n}({\cal Z}_n\setminus J),
 \end{equation}
 holds also for $J=\emptyset$ and any $m\in\N_0$. 
 
Theorem \ref{T1-main} has two notable consequences. The first corollary   
characterizes the finite sets for which the sequence of complete
homogeneous symmetric polynomials is periodic. The second corollary  shows that 
  every nonempty subset $J$ of $\mathcal Z_n$ determines a factor
$\mathcal P_J(z)$ of $z^n-1$, and that the coefficients of the
complementary factor are given by the complete homogeneous
symmetric polynomials evaluated on $J$.
  
  \begin{Cor}\label{C-periodic}
Let $n\in\N$ and let $J\subset\C \setminus\{0\}$ be finite and nonempty. Then 
$\H_r(J)=\H_{r+kn}(J)$ for every $0\leq r\le n-1$ and every $k\in \mathbb N_0$ if and only if $J\subseteq \cal Z_n$.
  \end{Cor}
  
  \medskip
 
 \begin{Cor}\label{C2bis}
 	Let $n\geq 2$ and let $J\subseteq \mathcal Z_n$ be  nonempty, with $m:=|J|$. We have that if $
 	  n-m<j<n$, then $
 	\H_j(J)=0$. Further, the polynomials
 $$
 	\mathcal P_J(z)
 	:=
 	\sum_{j=0}^{m}(-1)^j\E_j(J)z^{m-j}
 \quad 
  \hbox{and}
 \quad
 	\mathcal Q_J(z)
 	:=
 	\sum_{j=0}^{n-m}\H_j(J)z^{n-m-j}
 $$ 
 are such  that	
 $$
 	\mathcal Q_J(z)
 	=
 	\prod_{x_i\in\mathcal Z_n\setminus J}(z-x_i) \quad 
 	\hbox{and}
 	\quad \mathcal P_J(z)\mathcal Q_J(z)= z^n-1.
 	$$
 	\end{Cor}
 This corollary generalizes the   factorization
 $
 z^n-1=\Phi_n(z)\Psi_n(z),
 $
 where \begin{equation}\label{e-def-inverscy}
 	\Psi_n (z)  :=\prod_{\zeta_n^j\in\mathcal Z_n^{**}}(z-\zeta_n^j),\quad \hbox{with}\ \mathcal Z_n^{**}:=  \cal Z\setminus\cal Z^*,
 	\end{equation}    
is the $n$th inverse (or reciprocal) cyclotomic polynomial. 
  Thus, denoting
  \begin{equation}\label{defcyclos}
 \Phi_n(z)=\sum_{j=0}^{\varphi(n)}a_n(j)z^{\varphi(n)-j},
 \qquad
 \Psi_n(z)=\sum_{j=0}^{n-\varphi(n)}
 b_n(j)z^{n-\varphi(n)-j},
 \end{equation}
 where $ \varphi(n)=|\mathcal Z_n^*| $ is the Euler totient function, \eqref{e-classVieta} and \eqref{1main} yield
 \begin{align}\label{P-coeff-rec-pol0}
 	a_n(j) &= (-1)^j \E_j({\cal Z}_n^{* })=\H_j({\cal Z}_n^{* *}),\quad   j=0,1,\ldots,\varphi(n),
 	\\
 	\label{P-coeff-rec-pol}
 	b_n(j) &=(-1)^j \E_j({\cal Z}_n^{* *})=\H_j({\cal Z}_n^{* }),\   j=0,1,\ldots, n-\varphi(n).
 \end{align}
 These  identities were already established in
 \cite[Cor.\ 1.2]{DEL}.

 \medskip
The paper is  organized as follows.
 In this section, we apply Theorem \ref{T1-main} to derive new cyclotomic identities for the symmetric polynomials  $\H_m$ and $\E_m$, as well as new properties of the cyclotomic polynomials $\Phi_n $ and $\Psi_n $.
The proofs of our results are collected in \S\ref{Ths}. In \S\ref{preliminaries} we recall some properties of
$\H_m$ and $\E_m$, together with a generalization of the Vieta formulas \eqref{e-classVieta}.

 \subsection{Cyclotomy of  the complete homogeneous symmetric polynomials}\label{cyclotomy}
In this sub-section we state   new results  on the cyclotomy   of  the  polynomials $\H_m$ that  complement the results proved  in \cite{DEL, DL}.

 First, as a straightforward consequence  of Theorem \ref{T1-main}, we see that if $n\in\N$, with $n\ge 2$, and
		$m\in\N_0$, then \begin{align}\label{1e-spec-case-1}
		\H_m({\cal Z}_n\setminus\{1\})& =\H_m(\zeta_n, \zeta_n^2, \ldots, \zeta_n^{n-1})\nonumber\\ &=\H_{\{m\}_n}({\cal Z}_n\setminus\{1\})=\begin{cases} (-1)^{\{m\}_n}&  \mbox{if  $\{m\}_n=0, 1$,} 
			\cr  0&   \mbox{otherwise. }\cr\end{cases}
		\end{align}

	  In what follows,  we denote   $ 
	  \H_m(I,J)=\H_m(I\cup J)$ for any given finite disjoint sets  $I,J\subseteq\C$. 
	  		\begin{Cor}\label{C3} Given $n\in\N$, with $n\ge 2$, let $J\subset{\cal Z}_n$ be such that $m:=|J|\in (0, n)$.  The following properties hold true. 
	  		\begin{itemize}
	  			\item[(1)] 
	  			For every  $j\in \{n-m,\ldots, n-1\}$ and every $x\in \C\setminus\{0\}$ we have that
	  			$\H_j(J,x)=0$  if and only if  $x\in \cal Z_n\setminus J $.

	  			\item[(2)]  If $I \subseteq {\cal Z}_n\setminus J$, with $r:=|I|\not=0$, then  
	  			$ 
	  			\H_j(I,J)=0$ for every $j\in \{n-r-m+1,\ldots, n-1\}$.
	  		\end{itemize}
	  	\end{Cor}
	  \begin{Rem} From (1) it follows that the condition $\H_j(J,x)=0$, with $n-|J|\le j <n$ and $x\not=0$, yields   $x\in \cal Z_n$. Moreover, since \eqref{1main} implies that  $\H_j(J,0)=\H_j(J)=0$ when $n-m<j<n$, from the proof of (1) it is readily seen that
	  	$\H_j(J,x)=0$  if and only if  $x\in \{0\}\cup\cal Z_n\setminus J$ 
	  	when  $j\in \{n-m+1,\ldots, n-1\}$. 
	  	\end{Rem}

\subsection{Cyclotomy of the elementary  symmetric polynomials}
In this section we present new ciclotomic  identities for  the polynomials $\E_m$.

 \begin{Cor}\label{1P-Viete-1}
 	Let $k,n\in\N $ and $m\in\N_0$ such that 
 	 $m\leq k< n $. We have that
 		\begin{equation}\label{T-cyclic-1E}
 			\E_m (1, \zeta_n,\ldots, \zeta_n^{k-1})
 			=  \begin{cases}  \dsize(-1)^m \zeta_n^{mk}\prod_{j=1}^{n-k-1 }\frac{1-\zeta_n^{m +j }}{1-\zeta_n^j }& \mbox{if $k\leq n-2$,}\cr (-1)^m \zeta_n^{-m}&\mbox{if $k=n-1$.}
 				\end{cases}
 		\end{equation}
 	\end{Cor}
   We have only considered the case $m\leq k$ because   $	\E_m ( 1, \zeta_n,\ldots, \zeta_n^{k -1}) =0$  when $m>k$. The case $k=n$  is given by \eqref{e-classVietaz^n-1}. 
 \smallskip

The next identity  is the dual version of \eqref{1main} where the roles of $\E_m$ and $\H_m$ are exchanged. However, unlike  \eqref{1main}, it holds also when $J=\emptyset$.

\begin{Cor}\label{C4}
		  For every $m\in\N_0$ and every
	$J\subseteq {\cal Z}_n$  one has 
	\begin{equation}\label{eqC4}\hskip -.3 cm
		\E_m(J)=\begin{cases} (-1)^ {m}\H_{m}({\cal Z}_n\setminus J)+(-1)^ {m-1}\H_{m-n}({\cal Z}_n\setminus J)&  \mbox{if  $n\le m$ } 
			\cr  (-1)^ {m}\H_{m}({\cal Z}_n\setminus J)&   \mbox{otherwise. }\cr\end{cases} 
\end{equation}
\end{Cor}

\subsection{Cyclotomy involving Kronecker products}\label{kronecker}

\hskip 3 mm The Kronecker product of $ 
 \vec x=(x_1,\ldots,x_n)\in\C^n$ and $ \vec y=(y_1,\ldots,y_k) \in \C^k$ is 
 $$
 \vec x\otimes \vec y=(x_1\vec y,\ldots,  x_n\vec y) \in \C^{nk},
 $$
 where $x_j\vec y=(x_jy_1,\ldots ,x_jy_k)$ (see \cite{G}).  

Although the Kronecker product  is not commutative, it is plain that
 $
\H_m(\vec x\otimes\vec y)=\H_m(\vec y\otimes\vec x)$ and
$\E_m(\vec x\otimes\vec y)=\E_m(\vec y\otimes\vec x)$ for the vectors   $\vec x\otimes\vec y$ and $\vec y\otimes\vec x$   have the same multiset of components.
For this reason, in what follows we often identify vectors with the multiset of  their components.

 In this sub-section, we  consider  Kronecker products involving roots of unity. More precisely, 
the next theorem concerns   the evaluations of $\H_m$ and $\E_m$ at 
$\vec\xi\otimes\vec\zeta_n$, where $\vec\xi=(\xi_0,\ldots,\xi_{k-1})$ is any given vector and $ 
	\vec\zeta_n:=(1,\zeta_n,\ldots,\zeta_n^{ n-1})
	$ is 
the vector of the $n$th roots of unity. 

{\sl Notation.} We denote  $\vec \xi^{\,r}
:=(\xi_0^r,\ldots,\xi_{k-1}^r)$ when $r\in\N_0$.
More generally, this definition holds if $r$ is a real number such that each complex power $\xi_j^r$ is well defined.
In particular, if $r=s/h$, with $s,h\in\N$, then for any given 
$y\in\C$ we take $y^{1/h}$ as the principal $h$th root of $y$,
namely
$
y^{1/h}
=
|y|^{1/h}e^{i {\rm Arg}(y)/h}$, with 
$-\pi< {\rm Arg}(y)\leq\pi
$. Thus, we let $\vec \xi^{\, s/h}
:=\big((\xi_0^{1/h})^s,\ldots,(\xi_{k-1}^{1/h})^s\big)$.

   \begin{Thm}\label{1T-h-prod-1-clean}
   	Let $k,n\in\N$. The following identities hold true
   	for every $m\in\N_0$ and every
   	$ 
   	\vec\xi=(\xi_0,\ldots,\xi_{k-1})\in\C^k
   	$:
   		\begin{equation}\label{1e-spec-case1}
   			\H_m(\vec\xi\otimes\vec\zeta_n)
   			=
   			\begin{cases}
   				\H_{m/n}(\vec\xi^{\,n})
   				&\text{if $n$ divides $m$}, \\
   				0
   				&\text{otherwise,}
   			\end{cases}
   		\end{equation}
   		\begin{equation}\label{1e-spec-case2}
   			\E_m(\vec\xi\otimes\vec\zeta_n)
   			=
   			\begin{cases}
   				(-1)^{m-m/n}\E_{m/n}(\vec\xi^{\,n})
   				&\text{if $n$ divides $m\le kn$},\\
   				0
   				&\text{otherwise.}
   			\end{cases}
   		\end{equation}
   \end{Thm}
   
   Taking $k=1$ and $\vec\xi=1$ in \eqref{1e-spec-case1} and  \eqref{1e-spec-case2} yields \eqref{1e-spec-case} and \eqref{e-classVietaz^n-1}, respectively.

  As a consequence of Theorem \ref{1T-h-prod-1-clean}, we obtain 
   formulas for $\H_m(\vec\xi\otimes\vec\zeta'_n)$ and $\E_m(\vec\xi\otimes\vec\zeta'_n)$, where
    $$ 
   \vec\zeta_n'
   :=(\zeta_n,\ldots,\zeta_n^{ n-1})
   $$  
   is the vector whose components form the set ${\cal Z}_n\setminus\{1\}$. See also \eqref{1e-spec-case-1}.
   
 In what follows, $[t]$ denotes the  integer part of the real number $t$.
   
   \begin{Cor}\label{1T-h-prod-1-cleanCor}
   	Let $k,n\in\N$, with $n\ge 2$. 
   The following identity holds true for every $m\in\N_0$ and every
   $ 
   \vec\xi=(\xi_0,\ldots,\xi_{k-1})\in\C^k
   $:
   		\begin{equation}\label{e-spec-case2}
   			\H_m(\vec\xi\otimes\vec\zeta_n')
   			=
   			\sum_{j=0}^{[ m/n]}
   			(-1)^{m-jn}
   			\E_{m-jn}(\vec\xi\,)
   			\H_j(\vec\xi^{\, n}).
   		\end{equation}
   		Further, 	
   		\begin{enumerate}
   			\item 
   		if $0\le m\leq(n-1)k$, then
   		\begin{equation}\label{1e-em-zetaprime}
   			\E_m(\vec\xi\otimes\vec\zeta_n')
   			= 
   			\sum_{j=0}^{[ m/n]}
   			(-1)^{m-j}
   			\H_{m-jn}(\vec\xi\,)
   			\E_j(\vec\xi^ {\,n}).
   		\end{equation}
   		
   \item if $m>(n-1)k$, then
   			\begin{equation}\label{1e-em-zetaprime*}
   				\sum_{j=0}^{[M/n]}
   			(-1)^{m-j}
   			\H_{m-jn}(\vec\xi\,)
   			\E_j(\vec\xi^ {\,n})=0,
   		\end{equation}
   		where $M:=\min\{m,nk\}$.
   	\end{enumerate}
   \end{Cor}

\begin{Rem}
	\label{general1T-h-prod-1-clean} Note that taking  $k=1$ and $\vec\xi=1$ in \eqref{1e-em-zetaprime}  yields 
	$
	\E_m (\vec\zeta_n')=	(-1)^m 
	$
	for $m\leq n-1$,
	which   follows straightforwardly  from \eqref{1main}, or  equivalently from the Vieta formulas \eqref{e-classVieta} for $\dsize \frac{z^n-1}{z-1}=1+z+\ldots+z^{n-1}$.
\end{Rem}

\subsection{Applications to  cyclotomic  polynomials }\label{cyclopoly}
In what follows,
the letter $p$ (with or without subscript)  is reserved for the prime numbers whose set is
denoted  by $\Primes$.     

In this sub-section, after recalling some basic properties, we state our results on the coefficients $a_n(j), b_n(j)$  of the cyclotomic   polynomials $\Phi_n(z)$ and $\Psi_n(z)$ defined in   \eqref{e-def-inverscy} and \eqref{defcyclos}. We refer the reader to   \cite{ ABT, Moree, S, T}  for an overview of these polynomials.

  It is well known that $\Phi_n(z)$ is irreducible  over the rationals for any $n$. Further,  the coefficients $a_n(j), b_n(j)$ are integers for any $j$. 
  
Let us recall the following identities. 
 For every $m\in\N$ and every $p\in\Primes$ one has
 \begin{equation}\label{basiccyclo}
 	\Phi_{mp}(z)=\begin{cases} \Phi_m(z^p) & \mbox{if $p$ divides $m$,}
 		\cr\Phi_m(z^p)/\Phi_m(z)& \mbox{otherwise,}\cr\end{cases}
 \end{equation}
 \begin{equation}\label{basicinvcyclo}
 	\Psi_{mp}(z)=\begin{cases} \Psi_m(z^p) & \mbox{if $p$ divides  $m$,}
 		\cr\Psi_m(z^p) \Phi_m(z)& \mbox{otherwise.}\cr\end{cases}
 \end{equation}
 Observe that ${\rm deg}\Psi_{mp}(z)=mp-\varphi(mp)$, where 
 $$
 \varphi(mp)=\begin{cases} p\varphi(m)& \mbox{if $p$ divides $m$,}
 	\cr (p-1)\varphi(m)& \mbox{otherwise.}\cr\end{cases}
 $$
In  \cite[Th.\ 1.3]{DL}  we have obtained
relations between the coefficients of the cyclotomic polynomials $\Phi_{m}(z)$ and $\Phi_{mp}(z)$  by using  \eqref{basiccyclo} and our generalization of the Vieta formulas  \cite[Th.\ 1.1]{DEL} (see Proposition \ref{T-Viete-3} below).

 In the next theorem we summarize these relations and  add analogous formulas for the coefficients  of $\Psi_{m}(z)$ and $\Psi_{mp}(z)$.
 \vfill\eject
 
 \begin{Thm}\label{T-cross-prod-2Th}  Let $m\in\N$ and $p\in\Primes$.
 	\begin{enumerate}
 		\item If $p$ divides $m$, then 
 		\begin{equation}\label{newAMdiv1}	a_{mp}(k):=
 			\begin{cases} a_m(k/p) & \mbox{if $p$ divides $k$,}
 				\cr 0& \mbox{otherwise,}\cr\end{cases}
 		\end{equation}
 		for every $k\in\{0,1,\ldots,p\varphi(m)\}$ and 
 		\begin{equation}\label{newAMdiv2}
 			b_{mp}(k):=
 			\begin{cases} b_m(k/p) & \mbox{if $p$ divides $k$,}
 				\cr 0& \mbox{otherwise,}\cr\end{cases}\end{equation}
 		for every $k\in\{0,1,\ldots,mp- p\varphi(m)\}$. 
 		\item
 		If $p$ does not divide $m$, then  
 		\begin{equation}\label{1e-em-cycl}
 			a_{mp}(k)
 			=\sum_{j=0}^{[ k/p]}
 			a_m(j)b_m(k-jp) 
 		\end{equation}
 		for every $k\in\{0,1,  \ldots, (p-1)\varphi(m)\}$
 		and
 		\begin{equation}\label{newAM2}
 			\sum_{ j =0}^{k} b_{mp}(j)b_m(k-j)=
 			\begin{cases} b_m(k/p) & \mbox{if $p$ divides $k$}
 				\cr 0& \mbox{otherwise,}\cr\end{cases}
 		\end{equation}
 		for every $k\in\{0,1,\ldots,mp-p\varphi(m)\}$.
 	\end{enumerate}
 	
 \end{Thm}
 \smallskip

Let us recall the following  notation introduced in  \cite{DEL}.   

We define
$\vec\zeta_n^*:=(\zeta_{1,n}^*,\ldots ,\zeta_{j,n}^*,\ldots,\zeta_{n,n}^*)$, where  
	$$
	\zeta_{j,n}^*:=\begin{cases} \zeta_n^j  & \mbox{if ${\rm gcd}(j,n)=1$},\cr
		0& \mbox{otherwise.}
		\cr\end{cases}
	$$
Plainly, the nonzero components of $\vec\zeta_n^{* }\in\C^n$ are the $n$th primitive roots of unity which form the set ${\cal Z}_n^*$; thus, we can write
${\cal H}_k(\vec\zeta_n^*)={\cal H}_k({\cal Z}_n^*)$ and
${\cal E}_k(\vec\zeta_n^*)={\cal E}_k({\cal Z}_n^*)$.

 \begin{Prop}\label{L-Kron2}  
 	Let 	$n=p_1p_2\cdots p_t$ be squarefree.  For any given $k\in\N$, the roots of $\Phi_n(z^k)$ are the nonzero components of the Kronecker product
 	$$
 	\vec\zeta_{k}\otimes	(\vec\zeta_{p_1}^*)^{\frac 1k}\otimes\cdots \otimes (\vec\zeta_{p_t}^*)^{\frac 1k},
 	$$
 	where $\vec\zeta_{p_j}^*:=(0,\zeta_{p_j}, \zeta_{p_j}^2,\ldots,\zeta_{p_j}^{p_j-1})$.
 	\end{Prop}
 	The proof of this proposition is given in \S\ref{Ths}.

 	\begin{Rem} 
 		In view of \eqref{P-coeff-rec-pol0} and \eqref{P-coeff-rec-pol}, an immediate consequence of Proposition \ref{L-Kron2}  is that 
 		if $n=p_1p_2\cdots p_t$ is squarefree, then
 		\begin{align*}
 			a_n(j)= & (-1)^j \E_j(\vec\zeta_{p_1}^*\otimes\cdots \otimes \vec\zeta_{p_t}^*),\quad \forall j=0,1,\ldots,\prod_{s=1}^t(p_s-1),
 			\\
 			b_n(j) &=\H_j(\vec\zeta_{p_1}^*\otimes... \otimes \vec\zeta_{p_t}^*),\quad \forall j=0,1,\ldots, n-\prod_{s=1}^t(p_s-1).
 		\end{align*}
 	
 	\end{Rem} 
 			\begin{Rem}
 			Since \eqref{basiccyclo} yields
 			$\Phi_{m}(z)= \Phi_{\gamma(m)  }(z^{m/\gamma(m)} )$,
 			where
 			$$
 			\gamma(m):=
 			\begin{cases} 1& \mbox{if $m=1$},\cr
 				\prod_{p|m}p  & \mbox{if $m>1$,}
 				\cr\end{cases}
 			$$ 
 			is the so-called (squarefree) kernel of $m$, from Proposition \ref{L-Kron2} it follows that the roots of $\Phi_m(z)$, with $\gamma(m):=p_1p_2\cdots p_t$,  are 
 			the nonzero components of 
 			$$
 			\vec\zeta_{m/\gamma(m)}\otimes	(\vec\zeta_{p_1}^*)^{\gamma(m)/m}\otimes\cdots \otimes (\vec\zeta_{p_t}^*)^{\gamma(m)/m}.
 			$$
		In particular, if $p$ does not divide $m$, then $\gamma(mp)/mp=\gamma(m)/m$, so that ${\cal Z}_{mp}^*$ is the set of the nonzero components of the vector 	\begin{align*}\vec\zeta_{mp}^{*}&=
		\vec\zeta_{m/\gamma(m)}\otimes	(\vec\zeta_{p_1}^*)^{\gamma(m)/m}\otimes\cdots \otimes (\vec\zeta_{p_t}^*)^{\gamma(m)/m}\otimes (\vec\zeta_p^*)^{\gamma(m)/m}\\&=
		\vec\zeta_m^{*}\otimes (\vec\zeta_p^*)^{\gamma(m)/m}.
		\end{align*}
In view of \eqref{P-coeff-rec-pol}, \eqref{newAMdiv2}, and \eqref{newAM2}, we obtain the following
	 cyclotomic identities for $\H_m$:
	 	\begin{enumerate}
	 		\item 
	 		If $p\in\Primes$ divides $m$, then 
	 		$$
	 			\H_k (\vec\zeta_{mp}^{*}):=
	 			\begin{cases} \H_{k/p} (\vec\zeta_{m}^{*}) & \mbox{if $p\vert k$,}
	 				\cr 0& \mbox{otherwise,}\cr\end{cases}
	 		$$
	 		for every $k\in\{0,\ldots,mp-p\varphi(m)\}$.
	 		
	 		\item	If $p$ does not divide $m$, then
	 		\begin{equation}\label{newAMSim}
	 			\sum_{ j =0}^{k} \H_j  (\vec\zeta_m^{*}\otimes (\vec\zeta_p^*)^{\gamma(m)/m})\H_{k-j}(\vec\zeta_m^*)=
	 			\begin{cases} \H_{k/p} (\vec\zeta_{m}^{*}) & \mbox{if $p\vert k$}
	 				\cr 0& \mbox{otherwise,}\cr\end{cases}
	 		\end{equation}
	 		for every $k\in\{0,\ldots,mp-p\varphi(m)\}$.
	 	\end{enumerate}
	 	
Observe that by applying \eqref{1e- new-id-hm} of Lemma \ref{1L-new-id-hm} below to the left-hand side of \eqref{newAMSim} we obtain 
	$$
	\H_k  (\vec\zeta_m^{*}\otimes (\vec\zeta_p^*)^{\gamma(m)/m},\vec\zeta_m^*)=
	\H_k  (\vec\zeta_m^{*}\otimes \vec\zeta_p^{\,\gamma(m)/m})=
	\begin{cases} \H_{k/p} (\vec\zeta_{m}^{*}) & \mbox{if $p\vert k$}
		\cr 0& \mbox{otherwise,}\cr\end{cases}
	$$
	which has to be compared to \eqref{1e-spec-case1} after noticing that $(\vec\zeta_{m}^{*})^p$ and $\vec\zeta_{m}^{*}$ have the same components for $p$ does not divide $m$.
	
 \end{Rem}

\begin{Rem}
	Notice that if $\P_{\vec x}(z)$, $\P_{\vec y}(z)$, and $\P_{\vec x\otimes \vec y}(z)$ are
	the monic polynomials  whose  roots are  the components of $\vec x\in\C^n$, $\vec y\in\C^k$, and $\vec x\otimes \vec y$, respectively, then
	$$
	\P_{\vec x\otimes \vec y}(z)=\prod_{i=1}^{n}\prod_{j=1}^{k}(z-x_i y_j)
	= 	\prod_{i=1}^n x_i^{ k}\P_{\vec y}\Big(\frac{z}{x_i}\Big)
	=\prod_{j=1}^k y_j^{n} \P_{\vec x} \Big(\frac{z}{y_j}\Big).
	$$
In view of Proposition \ref{L-Kron2}, for $p_1\not=p_2$ this yields 
$$
\Phi_{p_1p_2}(z)=\prod_{j=1}^{p_1-1}\zeta_{p_1}^{j(p_2-1)}\Phi_{p_2}\Big(\frac{z}{\zeta_{p_1}^j}\Big)
=\prod_{j=1}^{p_2-1}\zeta_{p_2}^{j(p_1-1)}\Phi_{p_1}\Big(\frac{z}{\zeta_{p_2}^j}\Big).
$$  
\end{Rem}

 \subsection{Some explicit expressions for $\H_m({\cal Z}_n^{*})$ and $\E_m({\cal Z}_n^{*})$}
  In the next corollary we derive some explicit 
 expressions for
 $$
 \H_m({\cal Z}_{p_1p_2}^{*}),\ \E_m({\cal Z}_{p_1p_2}^{*}),\ \H_m({\cal Z}_{p_1p_2p_3}^{*}),\
 \E_m({\cal Z}_{p_1p_2p_3}^{*}),
 $$
 where $p_1,p_2,p_3$ are distinct primes.
 The identities \eqref{P-coeff-rec-pol0} and \eqref{P-coeff-rec-pol} show that such expressions are related to the coefficients of the so-called binary and ternary cyclotomic polynomials.
 We refer the reader to \cite{S} for a survey of what is known of these polynomials to date.
 Let us denote the indicator
 function of a set $A$ as
 $$ 
 \mathbf 1_A(x)
 :=
 \begin{cases}
 	1,&x\in A,\\
 	0,&x\notin A.
 \end{cases}
 $$
 
 \begin{Cor}\label{example}  
 	Let $m\in\N_0$ and let $p_1,p_2,p_3$ be primes such that $p_1<p_2<p_3$. For any given integer $j$, define 
 	$$
 	\Delta_j:=
 	\mathbf 1_{\big[0,[j/p_2]\big]}\bigl(b_{0,j}\bigr)
 	-
 	\mathbf 1_{\big[0,[j/p_2]\big]}\bigl(b_{1,j}\bigr),
 	$$
 	where $b_{0,j},b_{1,j}\in P_1:=\{0,1,\ldots,p_1-1\}$ are such that
 	$ 
 	p_2b_{0,j}\equiv j\pmod {p_1}$ and $
 	p_2b_{1,j}\equiv j-1\pmod {p_1}
 	$. We have that
 	\begin{enumerate}
 	\item[] \begin{equation}	\label{e-H2primes}
 		\H_m\big({\cal Z}_{p_1p_2}^{*}\big)=
 		\mathbf 1_{P_1}\big(\{m\}_{p_2}\big)
 		\H_{[ m/p_2]}\big({\cal Z}_{p_1}^{*}\big);
 \end{equation}
 \item[] \begin{equation}
 \label{e-E2primes}
 		(-1)^m\E_m\big({\cal Z}_{p_1p_2}^{*}\big)
 		=
 	\Delta_m\ \hbox{when}\ m\le (p_1-1)(p_2-1);
 \end{equation}
 Further,
 	\item[] \begin{equation}	\label{e-H3-explicit}
 		\H_m\big({\cal Z}_{p_1p_2p_3}^{*}\big)
 		=
 		\sum_{j\in I\cap\N_0}\mathbf 1_{P_1}\big(\{j\}_{p_2}\big)\Delta_{m-jp_3}\,\H_{[ j/p_2]}\big({\cal Z}_{p_1}^{*}\big)
 	\end{equation}
 	where $I:=\left[0,\big[\frac{m}{p_3}\big]\right]
 	\cap\left[\frac{m}{p_3}-\frac{(p_1-1)(p_2-1)}{p_3},+\infty\right)$;
 	\item[] \begin{align}	\label{e-E3-explicit}
 			&(-1)^m\E_m\big({\cal Z}_{p_1p_2p_3}^{*}\big)
 		=\nonumber\\
 		&\sum_{j=0}^{[ m/p_3]}
 		\mathbf 1_{P_1}\big(\{m-jp_3\}_{p_2}\big)
 		\mathbf 1_{P_1}\Big(\Big\{\big[\frac {m-jp_3}{p_2}\big]\Big\}_{p_2}\Big)\Delta_j\,
 		\H_{\big[[\frac {m-jp_3}{p_2}]/p_2\big]}\big({\cal Z}_{p_1}^{*}\big)
 	\end{align}
 		\ \hbox{when}\ $m\le (p_1-1)(p_2-1)(p_3-1)$.
 	\end{enumerate}
 	 \end{Cor}
 The proof is given in \S\ref{Ths}.

\begin{Rem} 
		From \eqref{1e-spec-case-1}, \eqref{e-H2primes}, and \eqref{e-E2primes} it follows that
	$$
	\H_j({\cal Z}_{p_1p_2}^{*}),\ \E_j({\cal Z}_{p_1p_2}^{*})\in\{-1,0,1\}.
	$$
In view of \eqref{P-coeff-rec-pol0} and \eqref{P-coeff-rec-pol},  this confirms the well known fact that 
	if $n$ is the product of two distinct prime factors, then
	the binary polynomials $\Phi_n(z)$ and $\Psi_n(z)$ are flat, i.e., the maximum absolute value of their coefficients is $1$. This result on $\Phi_n(z)$ was first established by Migotti in 1883 \cite{Mi}, then it was reproved and extended by various authors. We also have recently given a simple and direct proof in  \cite[Th.\ 1.4]{DL}. The reader is referred  to \cite{Moree} for the case of the binary inverse cyclotomic polynomial.

Concerning the ternary polynomials $\Phi_n(z)$ and $\Psi_n(z)$, where $n$ is the product of three distinct prime factors,  we observe that
 \eqref{e-H3-explicit} reduces to a
	particularly simple expression for $\H_{p_3}({\cal Z}_{p_1p_2p_3}^{*})$, which in some cases makes it easy to deduce that the ternary polynomials $\Psi_n(z)$ are not flat. More precisely, taking $m=p_3$ gives 
	$$
	I\cap\N_0=	\begin{cases} \{0\}&  \mbox{if  $p_3\ge (p_1-1)(p_2-1)$,} 
		\cr  \{0,1\}& \mbox{otherwise,}\cr
	\end{cases} 
	$$
	where $I$ is defined as in \eqref{e-H3-explicit}.
	Thus,  bearing in mind that $ 
	\E_0(\vec z)=\H_0(\vec z)=1$,
	the identity \eqref{e-H3-explicit} reduces to
	$$
	\H_{p_3}\big({\cal Z}_{p_1p_2p_3}^{*}\big)=
	\begin{cases} \Delta_{p_3}&  \mbox{if  $p_3\ge (p_1-1)(p_2-1)$,} 
		\cr  \Delta_{p_3}+\Delta_0& \mbox{otherwise.}\cr
	\end{cases} 
	$$
	where
	$$
	\Delta_0:=
	\mathbf 1_{\{0\}}\bigl(b_{0,0}\bigr)
	-
	\mathbf 1_{\{0\}}\bigl(b_{1,0}\bigr),\ \Delta_{p_3}:=
	\mathbf 1_{\big[0,[p_3/p_2]\big]}\bigl(b_{0,p_3}\bigr)
	-
	\mathbf 1_{\big[0,[p_3/p_2]\big]}\bigl(b_{1,p_3}\bigr).
	$$
	Recall that $b_{0,0},b_{1,0},b_{0,p_3},b_{1,p_3}\in \{0,1,\ldots,p_1-1\}$ are such that
	\begin{align*}
	p_2b_{0,0}&\equiv 0\pmod {p_1},\ \quad
	p_2b_{1,0}\equiv -1\pmod {p_1},\\
	p_2b_{0,p_3}&\equiv p_3\pmod {p_1},\quad
	p_2b_{1,p_3}\equiv p_3-1\pmod {p_1}.
	\end{align*}
	Consequently,
	$$
	\H_{p_3}\big({\cal Z}_{p_1p_2p_3}^{*}\big)=
	\begin{cases} \Delta_{p_3}&  \mbox{if  $p_3\ge (p_1-1)(p_2-1)$,} 
		\cr  \Delta_{p_3}+1& \mbox{otherwise.}\cr
	\end{cases} 
	$$
	In particular, this shows that  
	$ 
	\H_{p_3}\big({\cal Z}_{p_1p_2p_3}^{*}\big)\in\{0,1,2\}.
	$ 
	Further, assuming that $p_3<(p_1-1)(p_2-1)$, one has $	\H_{p_3}\big({\cal Z}_{p_1p_2p_3}^{*}\big)
	=2$ if and only if
	$
	b_{0,p_3}\le [p_3/p_2]<b_{1,p_3}$. 	In view of  \eqref{P-coeff-rec-pol}, this reveals that if $p_3< (p_1-1)(p_2-1)$, then  the inverse cyclotomic polynomial $\Psi_{p_1p_2p_3}(z)$ is not flat.
	
	We postpone  some analogous considerations for the ternary cyclotomic polynomial $\Phi_{p_1p_2p_3}(z)$ to Remark \ref{ternaryflat} below and refer to \cite{ Bz, Moree}  for related results on the coefficients of ternary inverse cyclotomic polynomials.
	
\end{Rem}

\section{Some properties of the symmetric polynomials}\label{preliminaries}

\subsection{Generating functions  }

The generating functions associated to the polynomials $\H_m(\vec z\,)={\cal H}_m (z_1,\ldots,z_n)$ and  $\E_m(\vec z\,)=	{\cal E}_m (z_1,\ldots,z_n)$ are respectively
$$
	H_n(t;\vec z\,):=\prod_{j=1 }^{n}\frac{1} {1-tz_j}=\sum_{m=0}^\infty\H_m(\vec z\,)t^m, 
$$
$$
	E_n(t;{\vec z}):=
	\prod_{j=1 }^n( 1+tz_j )=\sum_{m=0}^n\E_m(\vec z\,)t^m,
$$
where the series for $H_n(t;\vec z\,)$ is a formal power series.
It follows that 
\begin{equation}\label{e-gen-funct1}
	\H_m(\vec z\,)=\frac{1}{m!}\frac{{\rm d}^m}{{\rm d}t^m} H_n(t;\vec z\,)_{\vert_{t=0}},  \qquad
	\E_m(\vec z\,)=\frac{1}{m!}\frac{{\rm d}^m}{{\rm d}t^m} E_n(t;\vec z\,)_{\vert_{t=0}},
\end{equation}
for every $m\in\N_0$.
Since in the first formula one has the formal derivatives of $H_n(t;\vec z\,)$, it is worth 
recalling that both the Leibniz product rule and the chain rule hold also for the formal derivatives \cite{St}.

\subsection{Convolution identities}

The identities of the next lemma appear e.g. in  \cite[Th.\ 1.1]{Merca} (see also \cite{AE}) in a more general setting.

\begin{Lemma}\label{1L-new-id-hm}
	Let $m\in\N_0$ and $\vec v_1\in \C^{d_1},\ldots,\vec v_n\in \C ^{d_n}$.
	 We have that
		\begin{equation}\label{1e- new-id-hm}
			\H_m (\vec v_1,\dots,\vec v_n)=\sum_{m_1+\ldots+m_n=m\atop m_i\ge 0}\H_{m_1}(\vec v_{1})\cdots\H_{m_n}(\vec v_n).
		\end{equation}
		\begin{equation}\label{1e- new-id-em}
			\E_m (\vec v_1,\dots,\vec v_n)=\sum_{m_1+\ldots+m_n=m\atop{0\leq m_i\leq d_i} }\E_{m_1}(\vec v_{1})\cdots\E_{m_n}(\vec v_n).
		\end{equation} 
\end{Lemma}

\subsection{A cyclotomic formula for $\H_m$}

Here we recall the following identity implied by \cite[Cor.\ 1.5]{DEL}.

\begin{Lemma}\label{T-cyclic-1}
	Let  $J_k:=\{1,\zeta_n, \zeta_n^2, \ldots, \zeta_n^{k} \} \subseteq {\cal Z}_n$, with $k\in\N_0$. For every $  m\in\N_0$ we have that
	$$
	\hskip - 3mm\H_m(J_k)\! =\!\H_{\{m\}_n}(J_k) 
	\! =\!\begin{cases}  1 &  \mbox{   $k =0$ or $\{m\}_n=0$,}  
	\cr\dsize\prod_{s=1}^{k   }\frac{1-\zeta_n ^{\{m\}_n +s }}{1-\zeta_n ^s } &\mbox{    $1\leq  k  +\{m\}_n< n $,}\cr  0&   \mbox{  $ k  +\{m\}_n\ge n $. }\cr\end{cases}
$$

\end{Lemma}

Note that \eqref{1e-spec-case} follows by taking $J_k={\cal Z}_n$, i.e., $k=n-1$.

\subsection{Generalization of the Vieta formulas}
When counted with their multiplicity, the  roots of
a polynomial form a multiset. If $X_m:= \{x_1,\ldots, x_{m}\}$ is the multiset of  the roots of a complex polynomial $\P(z)$  of degree $m\in\N$, then the number ${\rm m}(x_j;X_m)$ of occurrences of $x_j$ in $X_m$
coincides with the multiplicity of $x_j$ as a root of 
$\P(z)$. 
 
Now, we state a  generalization of the classical Vieta formulas \eqref{e-classVieta} established in \cite[Th.\ 1.1]{DEL}. This provides necessary conditions on the coefficients of a complex polynomial when some of its roots are given.

\begin{Prop}\label{T-Viete-3} Let
	$X_m$ be the  multiset of  all the  roots of
	the polynomial $\P(z)=\sum_{j=0}^m a_j z^{m-j}$.
	If  $Y_n$ is a multiset, with $n:=|Y_n|$,  such that
	$1\le {\rm m}(y;Y_n)\le {\rm m}(y;X_m)$ for any $y\in Y_n$, 
	then
	$$
	\P(z)=  \big(c_{0}z^{m-n  }+c_{1}z^{n-m-1}+\ldots + c_{m-n-1 } z+c_{n-m  }\big){\cal Q}_n(z),
	$$
	where ${\cal Q}_n(z)$ is the monic polynomial whose multiset of roots is 
	$Y_n$ and
	\begin{equation*}  c_{k}=    \sum_{j=0}^k a_j \H_{k-j}(Y_n),\quad \mbox{  $k\in\{0,1,\ldots, m-n\}$.}
	\end{equation*}
\end{Prop}

 \section{Proofs of the results}\label{Ths}

 \begin{proof}[Proof of Theorem \ref{T1-main}] 
 If $n=1$, then $J={\cal Z}_n=\{1\}$ and \eqref{1main} is trivially true because of \eqref{convention1}. Now, let us  denote $s:=|J|\in\N$ and consider the case $n\ge 2$. The generating function of $\H_m(J)$, $m \in\N_0$, can be written as (see \S\ref{preliminaries})
 	\begin{align*}
 		H_s(t;J)&=\prod_{y\in J} \frac{1} {1-ty }=\prod_{y\in {\cal Z}_n} (1-ty)^{-1}\prod_{y\in {\cal Z}_n\setminus J} (1-ty)\\
 		&=H_n(t; {\cal Z}_n)
 		E_{n-s}(-t; {\cal Z}_n\setminus J),
 	\end{align*}
 	where $E_{n-s}(t,  {\cal Z}_n\setminus J )$ is the generating function of
 	$\E_r({\cal Z}_n\setminus J )$, $r\in\N_0$.
 	
 	Therefore, by applying \eqref{e-gen-funct1} and the Leibniz formula   we see that
 	\begin{align}\label{pre1main}
 		\H_m(J)&=
 		\frac{1}{m!}\frac{{\rm d}^m}{{\rm d}t^m}
 		\Big( H_n(t;{\cal Z}_n)
 		E_{n-s}(-t;{\cal Z}_n\setminus J)\Big)_{\vert_{t=0}}\nonumber\\
 		&=
 		\frac{1}{m!}\sum_{r=0}^m {m\choose r}\frac{{\rm d}^{m-r}}{{\rm d}t^{m-r}}   H_n(t;{\cal Z}_n)_{\vert_{t=0}}\frac{{\rm d}^r
 		}{{\rm d}t^r}  E_{n-s}(-t;{\cal Z}_n\setminus J)_{\vert_{t=0}}
 		\nonumber\\
 		&=\sum_{r=0}^m \frac{1}{(m-r)!r!}
 		\frac{{\rm d}^{m-r}}{{\rm d}t^{m-r}}   H_n(t;{\cal Z}_n)_{\vert_{t=0}}\frac{{\rm d}^r
 		}{{\rm d}t^r}  E_{n-s}(-t;{\cal Z}_n\setminus J)_{\vert_{t=0}}
 		\nonumber\\
 		&=\sum_{r=0}^{m} (-1)^r\E_r({\cal Z}_n\setminus J )\H_{m-r}({\cal Z}_n).
 	\end{align}
 	By using  \eqref{1e-spec-case} this reduces to
 	$ \dsize
 		\H_m(J)=\!\!\!\!\sum_{r=0\atop r\equiv m \Mod n}^{m} (-1)^r\E_r({\cal Z}_n\setminus J).
 	$ 
 Hence, \eqref{1main}  follows for $\E_r({\cal Z}_n\setminus J)=0$ when $r>|{\cal Z}_n\setminus J|$, and  the condition $r\equiv m \Mod n$, with $
 	0\le r\le |{\cal Z}_n\setminus J|<n$, is satisfied only if $r=\{r\}_n=\{m\}_n$.
 \end{proof} 
  
 \begin{Rem} Note that including $J=\emptyset$  in \eqref{pre1main} gives a  particular instance of the
 	well known 
 	identity (see \cite{Mcdonald})
 	$$
 	\sum_{j=0}^m(-1)^j{\cal E}_j(z_1,\ldots, z_n){\cal H}_{m-j}(z_1,\ldots, z_n)=0.
 	$$
 \end{Rem} 
 
 \begin{proof} [Proof of Corollary \ref{C-periodic}] If $J$ is a nonempty subset of $\cal Z_ n$, then \eqref{1main} and \eqref{1mainempty}  yield
 	$
 	\H_m(J) = \H_{\{m\}_n}(J) $ for all  $m\in\N_0$, i.e., $\H_m(J)=\H_{m+nk}(J)$ for every $k\in\N_0$.
 	
 	Conversely, given
 	$J:=\{x_1,\ldots,x_s\}\subset\C\setminus\{0\}$,
 	assume that there exists $n\in\N$ such that 
 	$ 
 	\H_{m+kn}(J)=\H_m(J)$ for every 
 	$0\leq m\le n-1$ and $k\in \mathbb N_0$. 
 	Using the generating function for $\H_m(J)$, we see that
 	$$
 	\prod_{j=1}^s(1-x_jt)^{-1}
 	=\sum_{m=0}^{\infty}\H_m(J)t^m
 	=\sum_{k=0}^{\infty}t^{kn}
 	\sum_{m=0}^{n-1}\H_m(J)t^m
 	=\frac{\displaystyle\sum_{m=0}^{n-1}\H_m(J)t^m}
 	{1-t^n}.
 	$$
 	Consequently,
 	$\dsize 
 	\prod_{j=1}^s(1-x_jt)$ divides $  1-t^n
 	$, so that every $1/x_j$, $j=1,\ldots,s$, must be a zero of
 	$1-t^n$, i.e., $x_j^n=1$.  
 \end{proof}
 
 \begin{proof}[Proof of Corollary \ref{C2bis}]
 	By Theorem \ref{T1-main}, we have that   $\H_j(J)=(-1)^j \E_j(\mathcal Z_n\setminus J)$  for every  $j\in\{0,1,\ldots,n-1\}$. In particular, it follows that $\H_j(J)=\E_j(\mathcal Z_n\setminus J)=0$ whenever    $  |\mathcal Z_n\setminus J|= n-m <j<n$.
 	
 	Thus, in order to prove the claimed properties of the polynomials 	$\mathcal P_J(z)$ and 	$\mathcal Q_J(z)$, it suffices to observe that  by using the Vieta formulas \eqref{e-classVieta} we can write  
 	$$
 	\mathcal P_J(z)=\prod_{x\in J}(z-x),\quad
 	\mathcal Q_J(z) =\sum_{j=0}^{n-m} (-1)^j\E_j(\mathcal Z_n\setminus J)z^{n-m-j}= \prod_{x\in \mathcal Z_n\setminus J}(z-x),
 	$$  
 	so that
 	$$\mathcal P_J(z)\mathcal Q_J(z)=\prod_{x\in \mathcal Z_n }(z-x)= z^n-1,
 	$$
 	as required.
 \end{proof}

 	 \begin{proof}[Proof of Corollary \ref{C3}]  (1) By applying  \eqref{1e- new-id-hm} of Lemma \ref{1L-new-id-hm} we can write
 	 	$$
 	 	\H_j(J,x)=	\sum_{s=0}^j\H_s(J) x^{j-s}\ \hbox{for every}\ j\in\N_0.
 	 	$$
 	 	The conclusion follows from Corollary \ref{C2bis} because $\H_j(J)=0$
 	 	when $n-m< j<n$, so that the previous equation reduces to
 	 	$
 	 	\H_j(J,x)=x^{j+m-n}\cal Q_J(x)$,  where 
 	 	$\cal Q_J(z)$ is the polynomial considered in Corollary \ref{C2bis}.

 	 	(2) It suffices to proceed by induction on $r$, where (1) provides  the base induction case for $r=1$. 
 	 	\end{proof}
 	 
 \begin{proof}[Proof of Corollary \ref{1P-Viete-1}] 
 	First, observe that $\{m\}_n=m$ for $0\le m\leq n-1$. Then,
 	by applying Theorem  \ref{T1-main} we see that
 	\begin{align*}
 		\E_m (1, \zeta_n,\ldots,\zeta_n^{k-1})&=  (-1)^m\H_m(\zeta_n^k,\dots,\zeta_n^{n-1})\\
 		&= (-1)^m \zeta_n^{mk}\H_m(1, \zeta_n,\ldots, \zeta_n^{n-k-1}).
 	\end{align*}
 	If $k=n-1$, then this reduces to 
 	$$
 		\E_m (1, \zeta_n,\ldots,\zeta_n^{n-2})=(-1)^m \zeta_n^{m(n-1)}\H_m(1)=(-1)^m \zeta_n^{-m}.
 	$$
 	If $k\le n-2$, then 
 	 we apply Lemma \ref{T-cyclic-1} to get  \eqref{T-cyclic-1E} completely   proved. 
  \end{proof}
  
  \begin{proof}[Proof of Corollary \ref{C4}] Since $\{m\}_n=m$ when $0\le m<n$, the identity \eqref{eqC4} is equivalent to \eqref{1mainempty} in this case. 
  
  Let us assume that $n\le m$. 
  If $n<m$, then $\E_m(J)=0$. Thus, in view of \eqref{convention1}, the identity \eqref{eqC4} is obviously true 
  when $J={\cal Z}_n$.
  On the other hand, if $J\not={\cal Z}_n$, then from \eqref{1main} and \eqref{1mainempty} one has
  	$\H_{m}({\cal Z}_n\setminus J)=\H_{\{m\}_n}({\cal Z}_n\setminus J)$
  and $\H_{m-n}({\cal Z}_n\setminus J)=
  \H_{\{m-n\}_n}({\cal Z}_n\setminus J)$. Since
  $\{m\}_n=\{m-n\}_n$ for $n<m$, we see that  $\H_{m}({\cal Z}_n\setminus J)=\H_{m-n}({\cal Z}_n\setminus J)$. Hence, \eqref{eqC4} is true also in this case. 
  
  Let us consider the remaining case $n=m\in\N$. If $J={\cal Z}_n$, then 
  \eqref{eqC4} follows after noticing that
  \eqref{e-classVietaz^n-1} yields $\E_m({\cal Z}_n)=\E_n({\cal Z}_n)=(-1)^{n+1}$, while $\H_{m}({\cal Z}_n\setminus J)=
  \H_n(\emptyset)=0$ and $\H_{m-n}({\cal Z}_n\setminus J)=\H_0(\emptyset)=1$ because of  \eqref{convention1}. 
  If $J\subset{\cal Z}_n$, i.e., ${\cal Z}_n\setminus J\not=\emptyset$, then  $\E_m(J)=\E_n(J)=0$ and
  $\H_{m}({\cal Z}_n\setminus J)=\H_{\{m\}_n}({\cal Z}_n\setminus J)=\H_0({\cal Z}_n\setminus J)=\H_{m-n}({\cal Z}_n\setminus J)=1$,
 yielding \eqref{eqC4} also in this last case.
 \end{proof}

 	 \begin{proof}[Proof of Theorem \ref{1T-h-prod-1-clean}]
 	 	By using  the identity \eqref{1e- new-id-hm} of Lemma \ref{1L-new-id-hm}   we write
 	 	\begin{align*}
 	 		\H_m(\vec\xi\otimes \vec\zeta_n)
 	 		&=\H_m(\xi_0\vec\zeta_n,\ldots,\xi_{k-1}\vec\zeta_n)\\
 	 		&=\sum_{j_0+\cdots+j_{k-1}=m\atop j_t\ge 0}\H_{j_0}(\xi_0\vec\zeta_n)\cdots \H_{j_{k-1}}(\xi_{k-1}\vec\zeta_n)\\
 	 		&=\sum_{j_0+\cdots+j_{k-1}=m\atop j_t\ge 0}\xi_0^{j_0}\cdots\xi_{k-1}^{j_{k-1}}\,
 	 		\H_{j_0}(\vec\zeta_n)\cdots \H_{j_{k-1}}(\vec\zeta_n).
 	 	\end{align*}
 	 	In view of \eqref{1e-spec-case},
 	 	the product $\H_{j_0}(\vec\zeta_n)\cdots \H_{j_{k-1}}(\vec\zeta_n)$ vanishes unless each $j_t$ is a multiple of $n$, with $t=0,1,\ldots,k-1$.
 	 	In other words, the whole sum is $0$ if $n$ does not divide $m$. Now, assuming that $j_t=nj_t'$ for each $t=0,1,\ldots,k-1$, by \eqref{1e-spec-case} we infer
 	 	\begin{align*}
 	 		\H_m(\vec\xi\otimes \vec\zeta_n)
 	 		&=\sum_{j_0'+\cdots+j_{k-1}'=m/n\atop j'_t\ge 0}\xi_0^{nj_0'}\cdots\xi_{k-1}^{nj_{k-1}'}\\
 	 		&=\H_{m/n}(\xi_0^n,\ldots,\xi_{k-1}^n)
 	 		=\H_{m/n}(\vec\xi^{\, n}).
 	 	\end{align*}
 	 	Hence, \eqref{1e-spec-case1} is proved. 
 	 	\smallskip
 	 
 	 Let us turn our attention to  \eqref{1e-spec-case2}. Since it is trivially true when $m>kn$, let us assume  that $m\le kn$. Arguing as before,
 	 we can apply \eqref{1e- new-id-em} to get
 	 	$$
 	 	\E_m(\vec\xi\otimes\vec\zeta_n)
 	 	=
 	 	\sum_{j_0+\cdots+j_{k-1}=m \atop{0\le j_t\leq n}}
 	 	\xi_0^{j_0}\cdots\xi_{k-1}^{j_{k-1}}\,
 	 	\E_{j_0}(\vec\zeta_n)\cdots\E_{j_{k-1}}(\vec\zeta_n).
 		$$
 	 	Now, from \eqref{e-classVietaz^n-1}  it follows that
 	 	 the above sum is $0$ when $m$ is not a multiple of $n$. On the other hand, if $n$ divides $m$, then the contribution of each $\E_{j_t}(\vec\zeta_n)$ to the above sum is zero unless $j_t=0$ or $j_t=n$. If this is the case, then
 	 	 exactly $m/n$ of the integers
 	 	$j_t$ must be equal to $n$, while the remaining $k-m/n$ must be equal to $0$. Thus, when $n$ divides $m$, in view of \eqref{e-classVietaz^n-1} we can write 	 	
 	 	$$
 	 		\E_m(\vec\xi\otimes\vec\zeta_n)
 	 		=\sum_{\substack{S\subseteq\{0,\ldots,k-1\}\\ |S|=m/n}}\
 	 		\prod_{j\in S}\big(\xi_j^{n} (-1)^{n-1}\big)=(-1)^{m-m/n}
 	 		\sum_{\substack{S\subseteq\{0,\ldots,k-1\}\\ |S|=m/n}}
 	 		\prod_{j\in S} \xi_j^{n}.
 	 		$$
 	 		Hence, \eqref{1e-spec-case2}    follows because by definition
 	 		$$
 	 		\E_{m/n}(\vec\xi^{\,n})=
 	 		\E_{m/n}(\xi_0^n,\ldots,\xi_{k-1}^n)
 	 		=\sum_{\substack{S\subseteq\{0,\ldots,k-1\}\\ |S|=m/n}}
 	 		\prod_{j\in S} \xi_j^{n}.
 	 	$$
 	 	The theorem is   
 	 	proved.
 	 	\end{proof}
 	 	 \begin{proof}[Proof of Corollary \ref{1T-h-prod-1-cleanCor}]
Let us consider the generating function of $\H_m(\vec\xi\otimes \vec\zeta_n')$ and write
\begin{align*}
	H_{k(n-1)}(t;\vec\xi\otimes \vec\zeta_n')&=
	\prod_{r=0 }^{k-1}\prod_{j=1 }^{n-1}
	\frac{1} {1-t\xi_r\zeta_n^j}\\
	&=\prod_{s=0 }^{k-1}(1-t\xi_s)
		\prod_{r=0 }^{k-1}\prod_{j=0 }^{n-1}
	\frac{1} {1-t\xi_r\zeta_n^j}\\
	&=E_k(-t;\vec\xi\,)H_{kn}(t;\vec\xi\otimes \vec\zeta_n).
\end{align*}
Thus, the identity \eqref{e-gen-funct1} and the Leibniz formula  yield
\begin{align*}
	\H_m(\vec\xi\otimes \vec\zeta_n')&=\frac{1}{m!}\frac{{\rm d}^m}{{\rm d}t^m}\left( E_k(-t;\vec\xi\,)H_{kn}(t;\vec\xi\otimes \vec\zeta_n)\right)_{\vert_{t=0}}\\
	&=\frac{1}{m!}	\sum_{j=0}^m{m\choose j}
	\frac{{\rm d}^j}{{\rm d}t^j}H_{kn}(t;\vec\xi\otimes \vec\zeta_n)_{\vert_{t=0}}\frac{{\rm d}^{m-j}}{{\rm d}t^{m-j}} E_k(-t;\vec\xi\,)_{\vert_{t=0}}\\
	&=\sum_{j=0}^m(-1)^{m-j}\E_{m-j}(\vec\xi\,)\H_{j}(\vec\xi\otimes \vec\zeta_ n)\\
	&=\sum_{j=0\atop j\equiv 0\Mod n}^m(-1)^{m-j}\E_{m-j}(\vec\xi\,)\H_{j/n}(\vec\xi^{\,n}),
\end{align*}
after applying \eqref{1e-spec-case1}. Hence, \eqref{e-spec-case2} is proved.
\smallskip

Similarly, the generating function of $\E_m(\vec\xi\otimes \vec\zeta_n')$ is
$$
E_{k(n-1)}(t;\vec\xi\otimes \vec\zeta_n')=  \prod_{i=0 }^{k-1}\prod_{j=1 }^{n-1}  (1+t\zeta_n^j\xi_i)=
H_k(-t;\vec\xi\,)E_{kn}(t;\vec\xi\otimes \vec\zeta_n).
$$ 
As before, from \eqref{e-gen-funct1} and the Leibniz formula we get
\begin{align*}
	\E_m(\vec\xi\otimes \vec\zeta_n')&=\sum_{j=0}^m(-1)^{m-j}\H_{m-j}(\vec\xi\,)\E_{j}(\vec\xi\otimes \vec\zeta_ n)\\
	&=\sum_{j=0}^M(-1)^{m-j}\H_{m-j}(\vec\xi\,)\E_{j}(\vec\xi\otimes \vec\zeta_ n)\\
	&=\sum_{j=0\atop j\equiv 0\Mod n}^M(-1)^{m-j/n}\H_{m-j}(\vec\xi\,)\E_{j/n}(\vec\xi^{\,n}),
\end{align*}
where $M:=\min\{m,nk\}$ and we have applied \eqref{1e-spec-case2}, together with the fact that $\E_{j}(\vec\xi\otimes \vec\zeta_ n)=0$ when $j>kn$.

Hence, \eqref{1e-em-zetaprime} follows immediately because
$M=m$ when $m\le k(n-1)$,  while
to obtain \eqref{1e-em-zetaprime*} it suffices to recall that 
$\E_m(\vec\xi\otimes \vec\zeta_n')=0$ when $m>k(n-1)$. The corollary is proved.
\end{proof} 

\begin{proof}[Proof of Theorem \ref{T-cross-prod-2Th}] (1) We prove only \eqref{newAMdiv2} because \eqref{newAMdiv1} can be proved in  a completely analogous way by using \eqref{basiccyclo} as in \cite[Th.\ 1.3]{DL}.
 	To this end, let us
 	 write $m':=m-\varphi(m)$ and 
 	\begin{align*}
 	\Psi_{m}(z^p)&= \sum_{k=0}^{m'}b_m(k) z^{p(m'-k)}\\
 	&=
 	\sum_{k=0\atop{k\equiv 0 \Mod p}}^{pm'}b_m(k/p) z^{p m'-k }=
 	\sum_{k=0}^{pm'}c_m(k) z^{p m'-k},
 	\end{align*}
 	where 
 	$ 
 	c_m(k):=
 	\begin{cases} b_m(k/p) & \mbox{if $p\vert k$}
 		\cr 0& \mbox{otherwise.}\cr\end{cases}
 	$ 
 	
 	\noindent
 	If $p|m$,  then $\varphi(mp)= p\varphi(m)$ and \eqref{newAMdiv2} follows from \eqref{basicinvcyclo}.
 	
 	(2)  By applying Proposition \ref{T-Viete-3}, together with \eqref{basiccyclo}, in \cite[Th.\ 1.3]{DL}
 	we have proved that   if $p\in\Primes$
 	does not divide $m$, then
 	$$
 		a_{mp}(k)=\sum_{j=0}^{[k/p]}a_{m }(j)\H_{k-jp}(\vec\zeta^*_{m})
 	$$
 	for every $k\in\{0,1,\ldots,(p-1)\varphi(m)\}$. Therefore, \eqref{1e-em-cycl} follows in view of \eqref{P-coeff-rec-pol}.

 	In order to prove \eqref{newAM2}, note that if $p$  does not divide $m$, then $\varphi(mp)= (p-1)\varphi(m)$ and
 	\eqref{basicinvcyclo} yields $\Psi_{mp}(z)=\Psi_m(z^p) \Phi_m(z)$. We can apply Proposition \ref{T-Viete-3}  to get
 	$$  
 	\sum_{ j =0}^{k} b_{mp}(j)\H_{k-j}(\vec\zeta_m^*)=c_m(k):=
 	\begin{cases} b_m(k/p) & \mbox{if $p\vert k$}
 		\cr 0& \mbox{otherwise,}\cr\end{cases} 
 	$$
 	for any integer  $k\in
 	[0,{\deg}\Psi_{mp}(z)-{\deg}\Phi_{m}(z)]=[0,mp-p\varphi(m)]$.
 	Hence, 
 	\eqref{newAM2} follows by using \eqref{P-coeff-rec-pol}.
 \end{proof}
 
 \begin{proof}[Proof of Proposition \ref{L-Kron2}] First,
 	let us assume that $k=1$ and show that 
 	$ \vec\zeta_n^*=\vec\zeta_{p_1}^*\otimes\cdots\otimes\vec\zeta_{p_t}^*$, i.e.,
 	the roots of $\Phi_n(z)$ are the nonzero components of $\vec\zeta_{p_1}^*\otimes\cdots\otimes \vec\zeta_{p_t}^*$.
 	
 	To this end, let us take a root $ \zeta_n^s$ of 
 	$\Phi_n(z)$, so that ${\rm gcd}(n, s)=1$. Clearly, if $s_j$ is the remainder of the division of $s$ by $p_j$ for $j=1,\ldots, t$, then ${\rm gcd}(s_j,p_j)=1$. By 
 	the Chinese remainder theorem we see that
 	$$
 	s\equiv \sum_{j=1}^ts_j\frac{n}{p_j}n'_j \Mod n,
 	$$ 
 	where 
 	$n'_j$ is the inverse of $n/p_j \Mod{p_j}$, i.e., $n'_jn/p_j\equiv 1 \Mod{p_j}$. Thus,
 	we can write 
 	$
 	\zeta_n^s
 	=
 	\zeta_{p_1}^{s_1n'_1}\cdots\zeta_{p_t}^{s_tn'_t}
 	$,
 	with $(s_jn'_j,p_j)=1$, revealing that $\zeta_n^s$ is a component of $\vec\zeta_{p_1}^*\otimes\cdots\otimes \vec\zeta_{p_t}^*$, as required.
 	
 	Now, let us consider the case $k\ge 2$.  Since $\Phi_n(x^k)=0$ if and only if  $x^k\in\mathcal Z_n^*$, it suffices to observe that for any given $\zeta_n^s\in {\cal Z}_n^*$
 	the $k$ solutions of
 	$z^k=\zeta_n^s$ 
 are 
 	$
 	\zeta_{n }^{s/k}\zeta_k^h$, with $h=0,\ldots,k-1$. Indeed, for what we have seen in the case $k=1$, it turns out that $
 	\zeta_{n }^{s/k}$ is a component of the vector 
 	$(\vec\zeta_{p_1}^*\otimes\cdots\otimes \vec\zeta_{p_t}^*)^{1/k}=(\vec\zeta_{p_1}^*)^{1/k}\otimes\cdots\otimes (\vec\zeta_{p_t}^*)^{1/k}$, while $\zeta_k^h$ is a component of $\vec\zeta_k$. Hence, 	$
 	\zeta_{n }^{s/k}\zeta_k^h$, with $h=0,\ldots,k-1$, are the components of the vector $\vec\zeta_k\otimes(\vec\zeta_{p_1}^*)^{1/k}\otimes\cdots\otimes (\vec\zeta_{p_t}^*)^{1/k}$,
 	as claimed.
 \end{proof}
 
 \begin{proof}[Proof of Corollary \ref{example}]  
 	First, observe that 
 	$(\vec\zeta_{p_1}^{*})^{p_2}=\vec\zeta_{p_1}^{*}$. Then, by Proposition \ref{L-Kron2} we see that
 	$(\vec\zeta_{p_1p_2}^{*})^{p_3}=(\vec\zeta_{p_1}^{*}\otimes
 	\vec\zeta_{p_2}^{*})^{p_3}=\vec\zeta_{p_1}^{*}\otimes
 	\vec\zeta_{p_2}^{*}$.
 	
 	Further,
 	Theorem \ref{T1-main} (see also \eqref{1e-spec-case-1} and Remark \ref{general1T-h-prod-1-clean}) yields
 	$$
 	\H_j(\vec\zeta_{p_1}^{*})
 	= \mathbf 1_{\{0\}}(\{j\}_{p_1})
 	-
 	\mathbf 1_{\{1\}}(\{j\}_{p_1})=
 	\begin{cases}
 		1,
 		&
 		\text{if $p_1|j$},
 		\\
 		-1,
 		&
 		\text{if $p_1|j-1$},
 		\\
 		0,&\text{otherwise},
 	\end{cases}
 	$$
 	$$
 	\E_j(\vec\zeta_{p_1}^{*})
 	=	(-1)^j\mathbf 1_{P_1}(j)=
 	\begin{cases}
 		(-1)^j,&\text{if $0\le j\le p_1-1$},\\
 		0,&\text{otherwise}.
 	\end{cases}
 	$$
 	where $P_1:=\{0,1,\ldots,p_1-1\}$.
 	Now, by applying 
 	Corollary
 	\ref{1T-h-prod-1-cleanCor} and the latter formula we can write
 	\begin{align*}
 		\H_m(\vec\zeta_{p_1p_2}^{*})=\H_m(\vec\zeta_{p_1}^{*}\otimes
 		\vec\zeta_{p_2}^{*})&=	\sum_{j=0}^{[m/p_2]}
 		(-1)^{m-jp_2}
 		\E_{m-jp_2}(\vec\zeta_{p_1}^{*})
 		\H_j(\vec\zeta_{p_1}^{*})\nonumber\\ 
 		&=	\sum_{j=0}^{[m/p_2]}
 		\mathbf 1_{P_1}(m-jp_2)
 		\H_j(\vec\zeta_{p_1}^{*}).
 	\end{align*} 
 	This yields \eqref{e-H2primes} because $p_2>p_1$ implies that the unique integer $j\in\big[0,[m/p_2]\big]$ such that
 	$m-jp_2\in P_1$, i.e., $ 
 	0\le m/p_2-j\le p_1/p_2<1
 	$, is
 	$ 
 	j=[m/p_2]
 	$, so that $m-jp_2=m-[m/p_2]p_2=\{m\}_{p_2}$.
 	
 	Similarly, from 	Corollary
 	\ref{1T-h-prod-1-cleanCor} we see that if $m\le \varphi(p_1p_2)=(p_1-1)(p_2-1)$, then $[m/p_2]<p_1-1$ and
 	\begin{align*}
 		\E_m(\vec\zeta_{p_1p_2}^{*})&=\E_m(\vec\zeta_{p_1}^{*}\otimes
 		\vec\zeta_{p_2}^{*})\\
 		&=
 		(-1)^m
 		\sum_{j=0}^{[m/p_2]}
 		\left(
 		\mathbf 1_{\{0\}}\big(\{m-jp_2\}_{p_1}\big)
 		-
 		\mathbf 1_{\{1\}}\big(\{m-jp_2\}_{p_1}\big)
 		\right).
 	\end{align*}
  	Therefore, \eqref{e-E2primes} follows because 
 	$b_{0,m}$ and $b_{1,m}$  are the unique solutions in $P_1$ of the congruences
 	$ 
 	p_2b_{0,m}\equiv m\pmod {p_1}$ and $
 	p_2b_{1,m}\equiv m-1\pmod {p_1},
 	$ 
 	respectively. 
 	
 	In order to prove \eqref{e-H3-explicit}, we apply
 	Corollary
 	\ref{1T-h-prod-1-cleanCor} again to write
 	\begin{equation*}\label{e-H3pr}
 		\H_m(\vec\zeta_{p_1p_2p_3}^{*})
 		=
 		\sum_{j=0}^{[m/p_3]}
 		(-1)^{m-jp_3}
 		\E_{m-jp_3}(\vec\zeta_{p_1}^{*}\otimes
 		\vec\zeta_{p_2}^{*})\H_j(\vec\zeta_{p_1}^{*}\otimes
 		\vec\zeta_{p_2}^{*}),
 		\end{equation*}
 	where 
 	$$
 	(-1)^{m-jp_3}\E_{m-jp_3}(\vec\zeta_{p_1}^{*}\otimes
 	\vec\zeta_{p_2}^{*})=\!
 	\begin{cases} 0&\!\!\!\mbox{if\!  $j\!\in\left[0,\big[\frac{m}{p_3}\big]\right]
 			\cap\left[0,\frac{m-\varphi(p_1p_2)}{p_3}\right)\!\cap\!\N_0$} 
 		\cr  \Delta_{m-jp_3}&\!\!\! \mbox{if\!  $j\!\in I\cap\N_0$,}\cr
 	\end{cases} 
 	$$
 	where 	 $I:=\left[0,\big[\frac{m}{p_3}\big]\right]
 	\cap\left[\frac{m}{p_3}-\frac{(p_1-1)(p_2-1)}{p_3},+\infty\right)$ and
 	$$
 	\Delta_{m-jp_3}:=
 	\mathbf 1_{\big[0,[m/p_2-jp_3/p_2]\big]}\bigl(b_{0,m-jp_3}\bigr)
 	-
 	\mathbf 1_{\big[0,[m/p_2-jp_3/p_2]\big]}\bigl(b_{1,m-jp_3}\bigr).
 	$$
 Thus, \eqref{e-H3-explicit} follows by using 
  \eqref{e-H2primes}.
  
 Analogously, to show \eqref{e-E3-explicit} we see that
 	\begin{equation}\label{e-E3pr}
 		\E_m(\vec\zeta_{p_1p_2p_3}^{*})
 			=
 		\sum_{j=0}^{[ m/p_3]}
 		(-1)^{m-j}
 		\H_{m-jp_3}(\vec\zeta_{p_1}^{*}\otimes
 		\vec\zeta_{p_2}^{*})\E_j(\vec\zeta_{p_1}^{*}\otimes
 		\vec\zeta_{p_2}^{*}),
 	\end{equation}
 where $m\le (p_1-1)(p_2-1)(p_3-1)$ yields $[m/p_3]<(p_1-1)(p_2-1)$.
 	Hence, \eqref{e-E3-explicit} follows 
 	by applying the  identities \eqref{e-H2primes} and  \eqref{e-E2primes}.
 	
 	The corollary is proved.
 \end{proof}
 \begin{Rem}\label{ternaryflat}
 Taking $m=p_3$ in \eqref{e-E3pr} gives 
 	$$
 	\E_{p_3}(\vec\zeta_{p_1p_2p_3}^{*})
 	=\E_1(\vec\zeta_{p_1}^{*}\otimes
 	\vec\zeta_{p_2}^{*})-\H_{p_3}(\vec\zeta_{p_1}^{*}\otimes
 	\vec\zeta_{p_2}^{*}),
 	$$
 	where \eqref{e-H2primes} and \eqref{e-E2primes} respectively yield
 		$$
 	\H_{p_3}(\vec\zeta_{p_1}^{*}\otimes
 	\vec\zeta_{p_2}^{*})
 	=
 	\mathbf 1_{[0,p_1-1]}\big(\{p_3\}_{p_2}\big)
 	\H_{[p_3/p_2]}(\vec\zeta_{p_1}^{*}),
 	$$
 	$$\E_1(\vec\zeta_{p_1}^{*}\otimes
 	\vec\zeta_{p_2}^{*})=-\Delta_1=
 	\mathbf 1_{\{0\}}\bigl(b_{1,1}\bigr)-
 	\mathbf 1_{\{0\}}\bigl(b_{0,1}\bigr)=	\mathbf 1_{\{0\}}\bigl(b_{1,1}\bigr)=1
 	$$
 	Consequently,
 	$ 
 	\E_{p_3}(\vec\zeta_{p_1p_2p_3}^{*})=	1-
 	\mathbf 1_{[0,p_1-1]}\big(\{p_3\}_{p_2}\big)
 	\H_{[p_3/p_2]}(\vec\zeta_{p_1}^{*})\in\{0,1,2\}
 	$.  In particular,
 	$ 
 	\E_{p_3}(\vec\zeta_{p_1p_2p_3}^{*})=2
 	$ 
 	if and only if
 	$$
 	\{p_3\}_{p_2}<p_1
 	\qquad\text{and}\qquad
 	\left\{\Big[\frac {p_3}{p_2}\Big]\right\}_{p_1}=1.
 	$$
 In view of \eqref{P-coeff-rec-pol0},	this confirms the well known fact that the ternary cyclotomic polynomials  do not  need to be flat. A particular case  is the classical counterexample 
 	for $n=105=3\cdot 5\cdot 7$ giving $a_{105}(7)=
 	-	\E_{7}(\vec\zeta_{105}^{*})=-2$.
 	
 \end{Rem}

\end{document}